\documentclass [12pt,reqno]{amsart}
\usepackage{amsmath}
\usepackage{amsthm}
\newtheorem{thm}{Theorem}[section]
\usepackage{amsfonts, amssymb}
\usepackage[utf8]{inputenc} 
\usepackage[english]{babel}
\usepackage{mathtools}
\usepackage{ mathrsfs }
\usepackage{enumerate}
\usepackage{ amssymb }
\usepackage{cite}

\newtheorem{theorem}{Theorem}

\newtheorem{Corollary}[theorem]{Corollary}
\newtheorem{Question}[thm]{Question}
\newtheorem*{theorem*}{Theorem}

\newtheorem{lemma}[thm]{Lemma}
\newtheorem{proposition}[thm]{Proposition}
\newtheorem{conjecture}[thm]{Conjecture}

\theoremstyle{definition}

\newtheorem{notation}[thm]{Notation}
\theoremstyle{definition}
\newtheorem{example}[thm]{Example}
\theoremstyle{definition}
\newtheorem{definition}[thm]{Definition}

\numberwithin{equation}{section}

\def\R{{\mathbb R}}
\def\E{{\mathbb E\,}}

\def\Z{{\mathbb Z}}
\def\n{{\mathbb N}}

\def\DD{{\mathcal D}}

\def\N{{\mathbb N}}

\def\wt{\widetilde}

\def\PP{{\mathcal P}}
\def\PPP{{\mathscr{P}}}
\def\TT{{\mathcal T}}

\def \EE{{\mathcal{E}}}
\def\<{\langle}
\def\>{\rangle}
\def \ee{{\epsilon}}
\def \dd{{\delta}}
\def \ss{{\sigma}}
\def \gg {{\gamma}}
\def \aa {{\alpha}}

\newcommand{\cat}{\operatorname{CAT}}

\begin{document}

\title{Markov Type constants, flat tori and Wasserstein spaces}


\begin{abstract}
Let $M_p(X,T)$ denote the Markov type $p$ constant at time $T$ of a metric space $X$, where $p \ge 1$. We show that $M_p(Y,T) \le M_p(X,T)$ in each of the following cases:
\begin{enumerate}[(a)]
\item{$X$ and $Y$ are geodesic spaces and $Y$ is covered by $X$ via a finite-sheeted locally isometric covering, \label{Acov}}
\item{$Y$ is the quotient of $X$ by a finite group of isometries,} \label{Agroup}
\item{$Y$ is the $L^p$-Wasserstein space over $X$.\label{Awas}}
\end{enumerate}
As an application of (\ref{Acov}) we show that all compact flat manifolds have Markov type $2$ with constant $1$. In particular the circle with its intrinsic metric has Markov type $2$ with constant $1$. This answers the question raised by S.-I. Ohta and M. Pichot.

Parts (\ref{Agroup}) and (\ref{Awas}) imply new upper bounds for Markov type constants  of the $L^p$-Wasserstein space over $\R^d$. 
These bounds were conjectured by A. Andoni, A. Naor and O. Neiman.
They imply certain restrictions on bi-Lipschitz embeddability of snowflakes into such Wasserstein spaces.

\end{abstract}

\author{Vladimir Zolotov}
\address[Vladimir Zolotov]{Steklov Institute of Mathematics, Russian Academy of Sciences, 27 Fontanka, 191023 St.Petersburg, Russia and Mathematics and Mechanics Faculty, St. Petersburg State University, Universitetsky pr., 28, Stary Peterhof, 198504, Russia.}
\email[Vladimir Zolotov]{paranuel@mail.ru}

\maketitle

\section{Introduction}

Let $X$ be a metric space, $p \ge 1$, and $T \in \n$. We denote by $M_p(X,T) \in [0, \infty)$ the  Markov type $p$ constant at time $T$ of $X$, see Definition \ref{MTT}. The Markov type $p$ constant of $X$, denoted by $M_p(X)$ is defined by  
$$M_p(X) = \sup_{T \in \n}{M_p(X,T)} \in [1,\infty].$$
We say that $X$ has Markov type $p$ if $M_p(X) < \infty$. 



Ball \cite{Ball} introduced the concept of Markov type in his study of the Lipschitz extension problem.
Major results in this direction were obtained later by Naor, Peres, Schramm and Sheffield \cite{NPSS}.
The notion of Markov type has also found applications in the theory of bi-Lipschitz embeddings \cite{BLMN,LMN}. 

It was shown in \cite{OP} that if $X$ is a geodesic metric space with $M_2(X) = 1$ then $X$ is nonnegatively curved in sense of Alexandrov. Ohta \cite{Ohta} showed that there exists an universal constant $M_A$ such that every nonnegatively curved Alexandrov space $X$ has Markov type $2$ with $M_2(X) \le M_A$. The best known bound $M_A \le \sqrt{1 + \sqrt{2} + \sqrt{4\sqrt{2} - 1}} = 2.08\dots$ is due to Andoni, Naor, and Neiman \cite{ANN}.

The authors of \cite{ANN} mentioned that it is plausible that $M_A = 1$. Thus we have the following question.

\begin{Question}[A. Andoni, A. Naor, O. Neiman \cite{ANN}]\label{Quest}
Is it true that ${M_2(X) = 1}$ for every Alexandrov space $X$ of nonnegative curvature?
\end{Question}

In the original paper \cite{Ball} (see also \cite{LMN}) Ball shows that Hilbert spaces have Markov type $2$ with constant $1$. Therefore $M_2(X) = 1$ for every subset $X$ of $L^2$. Thus convex subsets of Hilbert spaces are examples of geodesic spaces with Markov type $2$ constant $1$.
But it seems that no other examples of geodesic spaces having Markov type $2$ with constant $1$ are known.
The following two theorems allow us to expand the list of such examples.
Though the theorems are motivated by Question \ref{Quest}, they do not involve Alexandrov geometry.

\begin{theorem}\label{ThmCases}
Let $p \ge 1$, and let $X$, $Y$ be metric spaces such that at least one of the following holds:
\begin{enumerate}
\item{$X$ and $Y$ are geodesic spaces and $Y$ is covered by $X$ via a finite-sheeted locally isometric covering,} \label{FSC}
\item{$Y$ is a quotient of $X$ by a finite group of isometries,}\label{FGC}
\item{$Y$ is the $L^p$-Wasserstein space over $X$.}\label{WS}
\end{enumerate}
Then for every $T \in \n$ we have $M_p(X,T) \ge M_p(Y,T)$ and ${M_p(X) \ge M_p(Y)}$.
\end{theorem}
 
Theorem \ref{ThmCases} is a merger of propositions.
See Proposition \ref{MarkovCovering} for the case (\ref{FSC}), Proposition \ref{Markovquotient} for the case (\ref{FGC}) and Proposition \ref{WassMarkov} for the case  (\ref{WS}). 

Though the finiteness assumptions in Theorem \ref{ThmCases}(\ref{FSC}, \ref{FGC}) may seem unnatural they can not be dropped, see Example \ref{FiniteIsNes}. 


Recall that a Riemannian manifold $(M, g)$ is called \textit{flat} if it is locally isometric to the Euclidean space. 
As an application of Theorem \ref{ThmCases}(\ref{FSC}) we show that compact flat manifolds have Markov type $2$ with constant $1$.


\begin{theorem}\label{Flatis21}
Compact flat Riemannian manifolds have Markov type $2$ with constant $1$. In particular $M_2(S^1) = 1$, where $S^1$ is a circle with its intrinsic metric.
\end{theorem}

This answers the question raised by S.-I. Ohta and M. Pichot \cite{OP}, see also \cite{ANN}.

Let $X$ be an Alexandrov space of nonnegative curvature and $p = 2$. In each of the cases dealt with in Theorem \ref{ThmCases} the space $Y$ is also nonnegatively curved in the sense of Alexandrov (see \cite{BGP} Section 4.6 for the cases (1), (2), and \cite{SturmGMMS} Proposition 2.10.iv for the case (\ref{WS})). Hence we can consider Theorem \ref{ThmCases} as a supporting evidence for the affirmative answer to Question \ref{Quest}. 


For a metric space $X$, we denote by $\PPP_p(X)$ the $p$-Wasserstein space over $X$, see Section \ref{SecWass}. 
As a consequence of Theorem \ref{ThmCases}(\ref{FGC},\ref{WS}) we obtain the following upper bounds on Markov type $p$ constants for the $p$-Wasserstein space over Euclidean space $\R^d$.

\begin{Corollary}\label{WEucl}
For every $p \in (2 , \infty)$ and $T, d \in \N$ we have
\begin{enumerate}
 \item{$M_p(\PPP_p(\R^d), T) \le 16d^{\frac{1}{2} - \frac{1}{p}}p^{\frac{1}{2}}T^{\frac{1}{2} - \frac{1}{p}},$}\label{WpEucl}
 \item{$M_2(\PPP_p(\R^d)) \le 4d^{\frac{1}{2} - \frac{1}{p}}\sqrt{p-1}.$}  \label{W2Eucl}
\end{enumerate}
\end{Corollary}

The upper bound for Markov type $2$ constant of $\PPP(\R^d)$ given by Corollary \ref{WEucl}(\ref{W2Eucl}) imply certain extension theorem for partial Lipschitz maps from     
$\PPP(\R^d)$ into $\cat(0)$ spaces, uniformly convex Banach spaces or more generally metric spaces with metric Markov cotype $2$ 
(See  \cite[Theorem 1.11, Corollary 1.13]{MN}).

For a metric space $(X, d_X)$  and $\aa \in (0, 1]$, the metric space $(X,d^\aa_X)$ is called  the \textit{$\aa$-snowflake} of $(X,d_X)$.

 As observed by A. Andoni, A. Naor and O. Neiman (see \cite[Section 3]{ANN}) an upper bound for $M_p(\PPP_p(\R^d),T)$ implies certain restriction on the embeddability of snowflakes into $\PPP_p(\R^d)$. Namely we have the following corollary improving the estimate in \cite[Theorem 2]{ANN}.

\begin{Corollary}\label{Distortion}
For every $n > 1$ there exists an $n$-point metric space $X_n$ such that for every $\alpha \in (\frac{1}{2},1]$, every $p \in (2,\infty)$ and every $d \in \n$ the $\alpha$-snowflake of $X_n$ does not admit an embedding to $\PPP_p(\R^d)$ with the bi-Lipschitz distortion less then $Cd^{-\frac{1}{2} + \frac{1}{p}}p^{-\frac{1}{2}}(\log n)^{\alpha - \frac{1}{2}}$, where $C > 0$ is an absolute constant.
\end{Corollary}

Corollary \ref{Distortion} answers the question posed by A. Andoni, A. Naor and O. Neiman, see \cite[Question 23]{ANN}. 

\subsection*{Organization of the paper}
Definitions, preliminaries and notation are discussed in Section \ref{PrNt}.  
Lemmas for lifts of Markov chains are given in Section \ref{LoMC}.
Section \ref{SecCover} is devoted to results related to finite-sheeted locally isometric coverings. It contains the proof of Theorem \ref{ThmCases}(\ref{FSC}) (see Proposition \ref{MarkovCovering}) and the proof of Theorem \ref{Flatis21}. 
Section \ref{AoFG} is an analogue of Section \ref{SecCover} for quotients by finite groups of isometries,  it gives the proof of Theorem \ref{ThmCases}(\ref{FGC}), see Proposition \ref{Markovquotient}.
In Section \ref{SecWass} we present results related to Wasserstein spaces and a proof of Theorem \ref{ThmCases}(\ref{WS}), see Proposition \ref{WassMarkov}.
In Section~\ref{L3W} we prove Corollaries \ref{WEucl} and  \ref{Distortion}.
Section \ref{Counter} also contains counter examples, a conjecture and some additional results about lifts of Markov chains. 

\subsection*{Acknowledgements}
I thank my advisor Sergey V. Ivanov for all his ideas, advice and continuous support.
I am grateful to Prof. Assaf Naor for valuable comments on the preliminary version of the paper which result in particular in Corollary \ref{WEucl}(\ref{W2Eucl}).
 The paper is supported by the Russian Science Foundation under grant 16-11-10039.

\section{Definitions, preliminaries and notation}\label{PrNt}

Let $\{Z_t\}_{t=0}^{\infty}$ be a Markov chain on a finite state space $S$ with transition probabilities $a_{ij} = Pr[Z_{t+1} = j| Z_{t} = i]$, $i,j \in S$. The Markov chain $\{Z_t\}_{t=0}^{\infty}$ is said to be \textit{stationary} if $\pi_{i} = Pr[Z_t = i]$ does not depend on $t = 0,1,2,\dots$ 
A stationary Markov chain $\{Z_t\}_{t=0}^{\infty}$ is said to be \textit{reversible} if $\pi_ia_{ij} = \pi_j a_{ji}$, for every $i,j \in S$.

In order to construct a stationary reversible Markov chain on a finite set $S$, it suffices to define a nonnegative vector $(\pi_i)_{i \in S}$ and a nonnegative matrix $(a_{ij})_{i,j \in S}$ and verify that
\begin{enumerate}[(\thesection .1)]
\item{vector $\pi$ is stochastic, i.e. $\sum_{i \in S}\pi_i  = 1$,}\label{stvector}
\item{matrix $a$ is stochastic, i.e $\sum_{j \in S} a_{ij} = 1$, for every $i \in S$,} \label{stmatrix}
\item{$\pi_{i} a_{ij} = \pi_{j} a_{ji}$, for every $i, j \in S$.}\label{Estationary}
\end{enumerate}
The property (\ref{PrNt}.\ref{Estationary}) provides both stationarity and reversibility.  


Recall that a sequence of random variables $W = \{W_t\}_{t=0}^{\infty}$ on a set $X$ is called a \textit{random walk}.

\begin{definition}\label{DefMarkovWalk} \label{LiftChain}
We say that random walk $W$ on a set $X$ is a \textit{Markov walk} if there exists a stationary reversible Markov chain $\{Z_t\}_{t=0}^{\infty}$ on a finite state space $S$ and a map $f:S \rightarrow X$ such that $W_t = f(Z_t)$.

We say that Markov walks $W$ and $\wt W$ on a set $X$ are equivalent if the probability measures on the space of sequences $X^\infty = \{\{x_i\}_1^\infty|x_i \in X\}$  induced by $W$ and $\wt W$ coincide.
\end{definition}

\begin{notation}\label{notA}
Let $Z$ be a stationary reversible Markov chain on a finite state space $S$, and $S_0,\dots,S_T \subset S$. We denote by $A^{Z}(S_0,\dots,S_T)$ the probability $Pr[Z_0  \in S_0,\dots, Z_T \in S_T]$. For $s \in S$ and $S_1 \subset S$ we denote by $P^{Z}(s, S_1)$ the conditional probability $Pr[Z_1 \in S_1|Z_0 = s]$.
\end{notation}

For a Markov walk  $W$ on a metric space $X$ and  $T \in \n$,
we denote by $\EE_p(W, T)$ the expectation $\E d(W_T,W_0)^p$ . 


The following definition is a slightly reworded version of the one  in \cite[Section 3]{ANN}.


\begin{definition}\label{MTT}
Let $X$ be a metric space, $T \in \n$ and $p \ge 1$.  
The \textit{Markov type $p$ constant at time $T$} of $X$, denoted by $M_p(X,T)$ is  the infimum of all $K > 0$ such that for every Markov walk $W$  
on $X$,
$$
\EE_p(W, T) \le K^p T \EE_p(W, 1).
$$
The Markov type $p$ constant $M_p(X)$ is defined by 
$$M_p(X) = \sup_{T \in \n}{M_p(X,T)} \in [1,\infty].$$
We say that $X$ has Markov type $p$ if $M_p(X) < \infty$. 
\end{definition} 

%


For  a metric space $X$ we denote by $diam(X)$ the diameter of $X$,
and by $Iso(X)$ the group of isometries of $X$. 

Let $X$ and $Y$ be metric spaces, and $p \ge 1$. We write $X \times_p Y$ to denote the $p$-product space, i.e., the space with the distance 
defined by formula
$$d((x_1,y_1),(x_2,y_2))^p = d_X(x_1,x_2)^p + d_Y(y_1,y_2)^p.$$ 

We write $X^n_p$ to denote the $n$th $p$-power of $X$, i.e.,
$$X^n_p = X \times_p X \times_p \dots \times_p X\text{ ($n$ times)}.$$ 

The symmetric group $S_n$ acts on $X^n_p$ by permutation of coordinates. We denote by $X^n_p/S_n$ the corresponding quotient metric space. 
 
For $c >0$ we write $cX$  to denote the space with scaled metric, where the distance is defined by formula 
$$d_{cX}(x,y) = cd_{X}(x,y).$$ 

The following proposition immediately follows from the definitions 
\begin{proposition}\label{remMarkovProducts}
For $X$, $Y$ metric spaces, $p \ge 1$, $c > 0$, $n \in \n$ and  $T \in \n$ we have 
\begin{enumerate}
\item{$M_p(X \times_p Y, T) = \max\{M_p(X,T),M_p(Y,T)\}$,} 
\item{$M_p(X^n_p, T) = M_p(X,T)$,} 
\item{$M_p(cX, T) = M_p(X, T)$.} 
\end{enumerate} 
\end{proposition}

 
Let $U$, $V$ be two real-valued random variables. We write $U =_{st} V$, if $U$ and $V$ are equal in distribution.
For $U$ and $V$ defined on one probability space we write $U \overset{a.s.}{\le} V$ if $U \le V$ almost surely, i.e $Pr[U > V] = 0$. 

 
\begin{definition}\label{DefLiftMarkovWalk}
Let $X$ and $Y$ be two sets and let $\chi:X \rightarrow Y$ be a map. Let $\wt W$ and $W$ be Markov walks on $X$ and $Y$. We say that $\wt W$ is a \textit{lift} of $W$ along $\chi$, if Markov walks $\chi(\wt W)$ and $W$ are equivalent, see Definition \ref{DefMarkovWalk}. 

In the case when $X$ and $Y$ are metric spaces and in addition to the previous property we have $d(\wt W_1, \wt W_0) =_{st} d(W_1, W_0)$,  we say that $\wt W$ is a \textit{metric lift} of $W$ along $\chi$.

\end{definition}



\begin{proposition} \label{PropOfLift}
Let $X,Y$ be metric spaces and $\chi:X\rightarrow Y$ a short map.
Suppose that $\wt W$ is a metric lift of $W$ along $\chi$ then 
\begin{enumerate}
\item{$d(\wt W_1, \wt W_0) \overset{a.s.}{\le} diam(Y)$,} \label{Cdist} 
\item{$\frac{\EE_p(W, T)}{T\EE_p(W, 1)} \le \frac{\EE_p(\wt W, T)}{T\EE_p\wt W, 1)}$, for every $T \ge 2$ and every $p \ge 1$.}\label{LiftHasBiggerMC}
\end{enumerate}
\end{proposition}
\begin{proof}
The first claim follows from the definition of a metric lift.
The definition of a metric lift also implies that,
$$\EE_p(\wt W, 1) = \EE_p(W, 1)\text{, for every }p \ge 1.$$
From the definition of a lift and the fact that $\chi$ is a short map we have,
$$\EE_p(\wt W, T) \ge \EE_p(W, T)\text{, for every }T \ge 2\text{ and every }p \ge 1.$$
Thus implies the second claim.
\end{proof}

The plan of the proof of Theorem \ref{ThmCases}(\ref{FSC}) is to show that every Markov walk on the base space can be lifted to the covering space and apply Proposition \ref{PropOfLift}(\ref{LiftHasBiggerMC}) to the lift.

\begin{definition} \label{DefE}
For a stationary reversible Markov chain  $\{\wt{Z}_t\}_{t=0}^\infty$ on $\wt S$ we say that $\wt Z$ is 
\textit{restricted by} $E \subset \wt{S} \times \wt{S}$ if $A^{\wt Z}(\{x\},\{y\}) = 0$, for every $x,y \in \wt{S}$ such that $(x,y) \notin E$, see Notation \ref{notA}.
\end{definition}

Let $S$, $\wt{S}$ be finite sets, $E \subset \wt{S} \times \wt{S}$ be a symmetric subset and $\ss: \wt{S} \rightarrow S$ be a map. For $x \in \wt{S}$ and $V \subset \wt{S}$ we denote by $\deg_E(x,V)$ the number of elements of $\{y \in V : (x,y) \in E\}$. 
The following definition provides a condition on $E$ which implies that every stationary reversible Markov chain on $S$ admits a lift along $\ss$ restricted by $E$, see Lemma \ref{LiftMarkovChain}. 

\begin{definition}
We say that $\ss$ is \textit{regular with respect to} $E$ if 
$\deg_E(x,\ss^{-1}(s)) = \deg_E(y,\ss^{-1}(s)) \not = 0,$ for every $s \in S$ and every $x,y \in \wt S$ such that $\ss(x) = \ss(y)$.
\end{definition}

\section{Lifts of Markov chains}\label{LoMC}

The next lemma provides a sufficient condition for being  a lift of a Markov chain. A more complicated argument shows that this condition is also necessary, see Lemma \ref{SimpleChain}.


\begin{lemma}\label{LiftFromLoc}
Let $\{Z_t\}_{t=0}^\infty$ and $\{\wt Z_t\}_{t=0}^\infty$ be stationary reversible Markov chains on finite sets $S$ and $\wt S$ and $\ss:\wt{S}\rightarrow S$ a map such that,
\begin{enumerate}
\item{$A^{\wt Z}(\ss^{-1}(s)) = A^{Z}(\{s\})$, for every $s \in S$,}\label{0stepLFL}
\item{$P^{\wt Z}(\wt s_1, \ss^{-1}(s_2)) = P^{Z}(\ss(\wt s_1), \{s_2\})$, for every $\wt s_1 \in \wt S, s_2 \in S$.}\label{LocLFL}
\end{enumerate}
Then $\wt Z$ is a lift of $Z$ along $\ss$.
\end{lemma}
\begin{proof}
We have to show that for every $T \in \n$ and every $s_0,\dots,s_T \in S$, 
$$A^{\wt Z}(\ss^{-1}(s_0), \ss^{-1}(s_1), \dots,\ss^{-1}(s_T)) =A^{Z}(\{s_0\}, \{s_1\}, \dots, \{s_T\}).$$
The property (\ref{0stepLFL}) provides the case $T  = 0$. The general case follows from (\ref{LocLFL}) by induction.  
\end{proof}

The following lemma is a main technical tool of the paper. 

\begin{lemma}\label{LiftMarkovChain}
Let $S$, $\wt{S}$ be finite sets, $\ss:\wt{S}\rightarrow S$ a regular map with respect to a symmetric  set $E \subset \wt{S} \times \wt{S}$ and $\{Z_t\}_{t=0}^\infty$ a stationary reversible Markov chain on $S$. Then there exists a stationary reversible Markov chain $\{\wt{Z}_t\}_{t=0}^\infty$ on $\wt{S}$ such that 
{$\wt{Z}$ is a lift of $Z$ along $\ss$ and}
{$\wt Z$ is restricted by $E$.}

\end{lemma}
\begin{proof}
Let $\pi_x$, $a_{xy}$ be the stationary distribution and  the transition matrix for $Z_t$. For $x \in \wt S$ we denote $\ss^{-1}(\ss(x))$ by $M_x$.
We define a Markov chain $\wt{Z}$ by a distribution $\wt{\pi}_x = \frac{\pi_{\ss (x)}}{|M_x|}$ and a transition matrix 
$$\wt{a}_{xy} = 
\begin{cases}
  \frac{a_{\ss(x)\ss(y)}}{\deg_E(x,M_y)}, (x,y) \in E, \\
  0, \ (x,y) \not \in E. \\
\end{cases}$$


First we are going to show that $\wt \pi_x, \wt a_{xy}$ correctly define a stationary reversible Markov chain, i.e. to check the
properties (\ref{PrNt}.\ref{stvector})-(\ref{PrNt}.\ref{Estationary}). Properties (\ref{PrNt}.\ref{stvector}), (\ref{PrNt}.\ref{stmatrix}) and the case $(x,y) \not \in E$ of (\ref{PrNt}.\ref{Estationary})  follows directly from the definitions of $\wt \pi$ and $\wt a$.

In order to verify the case $(x,y) \in E$ of (\ref{PrNt}.\ref{Estationary}) we have to show that 
$\wt \pi_x \wt a_{xy} = \wt \pi_y \wt a_{yx}$ for every $x,y \in \wt S$ such that ${(x,y) \in E}$.
Fix $x,y \in \wt S$, let $N$ be the number of elements of the set ${(M_x \times M_y) \cap E}$. 
Since $\ss$ is regular with respect to $E$ we have 
$$|M_x| \deg_E(x,M_y) = N = |M_y| \deg_E(y,M_x).$$ 
Thus,
$$\wt \pi_x \wt a_{xy} = \frac{\pi_{\ss (x)}}{|M_x|}\frac{a_{\ss(x)\ss(y)}}{\deg_E(x,M_y)} = \frac{\pi_{\ss (x)}a_{\ss(x)\ss(y)}}{N} = $$
$$= \frac{\pi_{\ss (y)}a_{\ss(y)\ss(x)}}{N}  = \frac{\pi_{\ss (y)}}{|M_y|}\frac{a_{\ss(y)\ss(x)}}{\deg_E(y,M_x)} =  \wt \pi_y \wt a_{yx}.$$
As a result we have defined a stationary reversible Markov chain $\wt Z$.

Secondly, we have to show that $\wt Z$ is a lift of $Z$ along $\ss$.
From the definition of $\wt \pi$ we have
$$A^{\wt Z}(\ss^{-1}(s))  = \pi_s\text{, for every $s \in S$}.$$
and the definition of $\wt a$ provides 
$$P^{\wt Z}(x, \ss^{-1}(s)) = a_{\ss(x) s} \text{, for every $x \in \wt S, s \in S$}.$$ 
Applying Lemma \ref{LiftFromLoc} we obtain the claim.
\end{proof}


\section{Coverings and proof of Theorem \ref{Flatis21}}\label{SecCover}


The following lemma  implies Theorem \ref{ThmCases}(\ref{FSC}), see Proposition \ref{MarkovCovering}.

\begin{lemma}\label{GenLiftLemma}
Let $X,Y$ be geodesic spaces and $\chi: X \rightarrow Y$ a $k$-sheeted locally isometric covering. Then every Markov walk on $Y$ admits a metric lift along $\chi$ (see Definition \ref{DefLiftMarkovWalk}).

\end{lemma}

\begin{proof}
Let $W$ be a Markov walk on on $Y$ given by $W_t = f(Z_t)$, where $\{Z_t\}_{t=0}^\infty$ is a stationary reversible Markov chain on a finite set $S$ and $f$ is a map from $S$ to $Y$.


Define $\widetilde{S} = \{(s,x) \in S \times X: \chi(x) = f(s)\}$. We denote the projections from $\widetilde{S}$ to $S$ and $X$  by $\ss$  and $\widetilde{f}$.
For each unordered pair $\{s_1,s_2\}$ of (not necessary different) elements of $S$ fix a minimizing geodesic $\gg_{s_1s_2}$ connecting $f(s_1)$ and $f(s_2)$. 
Let $E$ be a set of all pairs $(x_1,x_2) \in \wt S \times \wt S$ such that there exists a lift of $\gg_{\ss(x_1)\ss(x_2)}$ connecting $\wt f(x_1)$ and $\wt f(x_2)$.
Note that for every $(x_1,x_2) \in E$,
\begin{equation}d_X(\wt f(x_1),\wt f(x_2)) = d_Y(f(\ss(x_1)),f(\ss(x_2))).\label{GLLshort}\end{equation} 

The existence and uniqueness of covering paths implies that ${\deg_E(x,\ss^{-1}(s)) = 1}$, for every $x \in \wt S,s \in S$. Hence $\ss$ is a regular map with respect to $E$. Lemma \ref{LiftMarkovChain} provides the existence of a stationary reversible Markov chain $\wt Z_t$ on $\wt S$, such that 
\begin{enumerate}
\item{$\wt{Z_t}$ is a lift of $Z_t$ along $\ss$,}\label{GLLLift}
\item{$\wt Z$ is restricted by $E$ (see Definition \ref{DefE}).}\label{respectE}
\end{enumerate}

We define $\wt W$ by $\wt W_t = \wt f (\wt Z_t)$. The definitions of $\wt S$ and $\wt f$ imply that ${\chi  \circ  \wt f =  f \circ \ss}$. Hence, the equivalence of $\ss(\wt Z)$ and $Z_t$ implies the equivalence 
of $\chi(\wt W)$ and $W$. Finally $\wt W$ is a metric lift of $W$, which follows from properties (\ref{GLLLift}), (\ref{respectE}) and (\ref{GLLshort}).
\end{proof}

\begin{proposition}\label{MarkovCovering}
Let $X,Y$ be geodesic spaces. Let $\chi:X \rightarrow Y$ be a finite sheeted locally isometric covering. Then for every $p \ge 1$ and $T \in \n$ we have $M_p(X,T) \ge M_p(Y,T)$ and ${M_p(X) \ge M_p(Y)}$.
\end{proposition}
\begin{proof}
The statement follows from Lemma \ref{GenLiftLemma} and Proposition \ref{PropOfLift}.
\end{proof}


\begin{proof}[Proof of Theorem \ref{Flatis21}]
Let $X$ be a compact flat Riemannian manifold. By Bieberbach's Theorem \cite{Bie1,Bie2}, $X$ can be covered by a flat torus.
Thus by Proposition \ref{MarkovCovering} it suffices to consider the case $X = \TT^d$, where $\TT^d$ is a flat torus of dimension $d$.





Let $W$ be a Markov walk on $\TT^d$.
For every positive integer $k$ the flat torus $\TT^d$ admits a locally isometric $k^d$-sheeted covering by the scaled torus $k\TT^d$. 
Indeed, if $\TT^d = \R^d/\Gamma$ where $\Gamma$ is a lattice, then $k\TT^d = \R^d/k \Gamma$.
The natural quotient map $\R^d/k \Gamma \rightarrow \R^d/\Gamma$ is a desired covering map.

 By Lemma \ref{GenLiftLemma}, Proposition \ref{PropOfLift} and rescaling there exists a Markov walk $W^k$ on $\TT^d$ such that 
\begin{equation}{d(W^k_1, W^k_0) \overset{a.s.}{\le} \frac{diam(\TT^d)}{k},}\end{equation}
\begin{equation}{\label{FR}\frac{\EE_2(W, T)}{T\EE_2(W, 1)} \le \frac{\EE_2(W^k, T)}{T\EE_2(W^k, 1)}.}\end{equation}



By the Nash embedding theorem (see \cite{Nash}) there exists a (Riemannian) isometric $C^1$-map $\Phi:\TT^d \rightarrow \R^{2d}$. Then for every $\ee > 0$ there exists $\dd(\ee) > 0$ such that 
\begin{equation}d(x,y) \le (1+\ee)||\Phi(x)-\Phi(y)||,\label{aliso}\end{equation}
 for every pair of points $x,y \in \TT^d$ with $d(x,y) < \dd(\ee)$.

Fix $\ee > 0$ and $T \in \N$. Choose $k > \frac{diam(T^d)T}{\dd(\ee)}$. Then $d(W_1^k, W_0^k) < \dd(\ee)/T$. Hence,

\begin{equation}\EE_2(W^k, T) \overset{(\ref{aliso})}{\le} (1+\ee)^2\EE_2(\Phi \circ W^k, T) \le
 (1+\ee)^2T\EE_2(\Phi \circ W^k, 1) \le (1 +\ee)^2 T \EE_2(W^k, 1),\label{AlmostMarkov}\end{equation}
where the second inequality follows from $M_2(\R^{2d}) = 1$. Thus,

$$\frac{\EE_2(W, T)}{T\EE_2(W, 1)} \overset{(\ref{FR})}{\le} \frac{\EE_2(W^k, T)}{T\EE_2(W^k, 1)} \overset{(\ref{AlmostMarkov})}{\le} (1 +\ee)^2.$$

Since $\ee$ is arbitrary, it follows that 
$\EE_2(W, T) \le T \EE_2(W, 1).$   
Thus ${M_2(\TT^d) = 1}$ and Theorem \ref{Flatis21} follows.
\end{proof}


\section{Quotients by finite groups}\label{AoFG}


Recall that, a finite group $G$ acting by isometries on a metric space $X$ induces a quotient metric on $X/G$, given by ${d_{X/G}(\bar{x},\bar{y}) = \min_{x \in \bar{x},y \in \bar{y}}d_X(x,y)}$. The following lemma is an analogue of Lemma \ref{GenLiftLemma} for quotient maps.

\begin{lemma}\label{quotientLiftLemma}
Let $X$ be a metric space. Let $G$ be a finite subgroup of $Iso(X)$, and let $\chi: X \rightarrow X / G$ be the corresponding quotient map. Then every Markov walk on $ X / G$ admits a metric lift along $\chi$.
\end{lemma}

\begin{proof}
The proof is similar to that of Lemma \ref{GenLiftLemma}, the only difference is the construction of the set $E$. 
Let $W_t$ be a Markov walk on on $X / G$ given by $W_t = f(Z_t)$, where $\{Z_t\}_{t=0}^\infty$ is a stationary reversible Markov chain on a finite set $S$ and $f$ is a map from $S$ to $X/G$.


Define $\widetilde{S} = \{(s,x) \in S \times X: \chi(x) = f(s)\}$. We denote the projections from $\widetilde{S}$ to $S$ and $X$  by $\ss$  and $\widetilde{f}$.
Let $E$ be a set of all pairs $(x_1,x_2) \in \wt S \times \wt S$ such that $d_X(\wt f(x_1),\wt f(x_2)) = d_{X/G}(f(\ss(x_1)),f(\ss(x_2)))$. 

Let  $s_1, s_2 \in S$ and $x_1,x_2 \in \ss^{-1}(s_1)$.   Since $\ss^{-1}(s_1)$ and $\ss^{-1}(s_2)$ are orbits of an isometric action of a finite group, we have ${\deg_E(x_1,\ss^{-1}(s_2)) = \deg_E(x_2,\ss^{-1}(s_2)) \not = 0}$. Hence $\ss$ is a regular map with respect to $E$. 

The rest of the proof is the same as in Lemma \ref{GenLiftLemma}.
\end{proof}

\begin{proposition}\label{Markovquotient}
Let $X$ be a metric space and $G$ be a finite subgroup of $Iso(X)$. Then for every $p \ge 1$ and every $T \in \n$ we have
$M_p(X,T) \ge M_p(X/G,T)$ and $M_p(X) \ge M_p(X/G)$.
\end{proposition}
\begin{proof}
The statement follows from Lemma \ref{quotientLiftLemma} and Proposition \ref{PropOfLift}.
\end{proof}

\section{Wasserstein spaces} \label{SecWass}

For reader's convenience we recall the definition of Wasserstein spaces. For further details see \cite{vil}. 

Let  $X$ be a metric space.
Let $p \ge 1$ and let $\mu$, $\nu$ be Borel probabalistic measures with finite $p$-th moment, i.e 
 $$\int_{X}d^p(x,o)d\mu(x) < \infty, \int_{X}d^p(x,o)d\nu(x) < \infty,$$
for some (hence all) $o \in X$.
We say that measure $q$ on $X \times X$ is a \textit{coupling} of $\mu$ and $\nu$ iff its marginals are $\mu$ and $\nu$, that is, iff 
$$q(A \times X) = \mu(A)\text{, }q(X \times A) = \nu(A),$$
for all Borel measurable subsets $A \subset X$.  
The $L^p$-Wasserstein distance between $\mu$ and $\nu$ is defined by 
$$d_{W_p}(\mu,\nu) = \inf\Big\{\Big(\int_{X \times X}{d^p(x,y)dq(x,y)}\Big)^{\frac{1}{p}}: \text{$q$ is a coupling of $\mu$ and $\nu$}\Big\}.$$
The  $L^p$-Wasserstein space $\PPP_p(X)$ is the set of Borel probabilistic measures with finite $p$-th moment on $X$ equipped with $L^p$-Wasserstein distance.

Recall that, for a metric space $X$ we denote by $X^n_p$ the $p$-power of $X$ and by $X^n_p/S_n$ the quotient space of $X^n_p$ by permutations of coordinates.
The following lemma allows to deduce Theorem \ref{ThmCases}(\ref{WS}) from Proposition \ref{Markovquotient}, see Proposition \ref{WassMarkov}.

\begin{lemma}\label{embed}
Let $X$ be a metric space, $n \in \Z$ and $p \ge 1$. The map ${\Phi_n:{n^{-\frac{1}{p}}}(X^n_p/S_n) \rightarrow \PPP_p(X)}$ defined by 
$$\Phi_n(x_1,\dots,x_n) = \frac{1}{n}\delta(x_1) + \dots + \frac{1}{n}\delta(x_n)$$
is a distance preserving map.
\end{lemma}
\begin{proof}
We denote ${n^{-\frac{1}{p}}}(X^n_p/S_n)$ by $Y$.
Fix two points $w = (w_1,\dots,w_n)$ and $q = (q_1,\dots,q_n)$ in $Y$. The distance between $w$ and $q$ is given by 
$$d^p_{Y}(w,q) = \frac{1}{n}\inf_{s \in S_n}\sum_{i=1}^nd^p(w_i,q_{s(i)}).$$
Let $\PP_n$  denote  the set of all $n \times n$ permutation matrices and $D$ the $n \times n$ matrix defined by $D_{ij} = d^p(w_i,q_j)$.   The formula for the distance  can be rewritten as 
$$d^p_{Y}(w,q) = \frac{1}{n}\inf_{A \in \PP_n} D \circ A,$$
where $\circ$ denotes the Hadamard product (entrywise product) of matrices. 

Let $\DD$ denotes the set of all $n \times n$ doubly stochastic matrices. Then the Wasserstein distance between $\Phi_n(w)$ and $\Phi_n(q)$ can be written as
$$d^p_{W_p}(\Phi_n(w),\Phi_n(q)) = \frac{1}{n}\inf_{A \in \DD} D \circ A.$$ 
By the Birkhoff-von Neumann theorem  $\DD$ is the convex hull of $\PP_n$. Since the  $"D \circ"$ is a linear functional it follows that    
$$ \frac{1}{n}\inf_{A \in \DD} D \circ A = \frac{1}{n}\inf_{A \in \PP_n} D \circ A.$$ 
Thus $d_Y(w,q) = d_{W_p}(\Phi(w),\Phi(q))$.
\end{proof}





\begin{proposition}\label{WassMarkov}
Let $X$ be a a metric space, $p \ge 1$ and $T \in \n$. Then ${M_p(\PPP_p(X), T) = M_p(X, T)}$ and ${M_p(\PPP_p(X)) = M_p(X)}$. 
\end{proposition}
\begin{proof}
For $k \in \n$ we denote by $I_k$ the image of $\Phi_{2^k}$,  where $\Phi_{2^k}$ is the map defined in Lemma \ref{embed}. Note that $I_k \subset I_{k+1}$ for every $k \in \n$.
Since the union $\cup_{k = 1}^{\infty} I_k$ is dense in $\PPP_p(X)$ (see \cite{vil}) we have 
$$M_p(\PPP_p(X), T) = \sup_{k \in \n}M_p(I_k, T).$$
From Lemma \ref{embed}, Proposition \ref{Markovquotient} and Proposition \ref{remMarkovProducts} it follows that 
$$M_p(I_k, T) = M_p({{(2^k)}^{-\frac{1}{p}}}(X^{(2/^k)}_p/S_{2^k}), T) \le M_p(X, T).$$
Hence we have $M_p(\PPP_p(X), T) \le M_p(X, T)$. The existence of isometric copy of $X$ in $\PPP_p(X)$ implies the opposite inequality.
\end{proof}

\section{Proofs of Corollaries \ref{WEucl} and \ref{Distortion} and counter examples.}\label{L3W}
\label{Counter}
\begin{proof}[Proof of Corollary \ref{WEucl}(\ref{WpEucl})]
For $p > 2$ and $T \in \n$ we have the following upper bound for $M_p(\R,T)$, 
$$M_p(\R,T) \le 16p^{\frac{1}{2}}T^{\frac{1}{2} - \frac{1}{p}},$$ 
see \cite [Theorem 4.5]{NPSS}.
Proposition \ref{remMarkovProducts} implies that $$M_p(\R^d_p,T) \le 16p^{\frac{1}{2}}T^{\frac{1}{2} - \frac{1}{p}}.$$ 
Since the $l_p$ norm on $\R^d$ is $d^{\frac{1}{2} - \frac{1}{p}}$-equivalent to the $l_2$ norm on $\R^d$, we obtain  $${M_p(\R^d,T) \le 16d^{\frac{1}{2} - \frac{1}{p}}p^{\frac{1}{2}}T^{\frac{1}{2} - \frac{1}{p}}}.$$ 
Finally Proposition \ref{WassMarkov} provides an upper bound  for $M_p(\PPP_p(\R^d),T)$,
   $${M_p(\PPP_p(\R^d),T) \le 16d^{\frac{1}{2} - \frac{1}{p}}p^{\frac{1}{2}}T^{\frac{1}{2} - \frac{1}{p}}}\text.$$
\end{proof}

\begin{definition}[see \cite{ANN}]\label{defdist}
Let $X$ and $Y$ be metric spaces and ${D \in [1 , \infty]}$.  A mapping $f : X \rightarrow Y$ is said to have distortion at most $D$, if there exists $s \in (0, \infty )$ such that every $x,y \in X$ satisfy $${sd_X(x, y) \le d_Y(f(x), f(y)) \le Dsd_X(x, y)}.$$  The infimum over those $D \in [1, \infty]$ for which this holds true is called the \textit{distortion} of $f$ and is denoted by $dist(f)$.  The infimum of $dist(f)$ over all $f:X \rightarrow Y$ is denoted by $c_{Y}(X,d_X)$.
\end{definition}

We remind the reader that for a metric space $(X, d_X)$ and $\aa \in (0, 1]$, the metric space $(X,d^\aa_X)$ is called  the $\aa$-snowflake of $(X,d_X)$. 

\begin{proof}[Proof of Corollary \ref{Distortion}]

The following lemma provides a restriction on bi-Lipschitz embeddability of snowflakes into spaces with bounded Markov type constants. 

\begin{lemma}[\cite{ANN}, Lemma 16] \label{Hamm}
Fix a metric space $Y$, $T \in \n$, $K, p \in [1, \infty)$ and $\zeta \in [0,1]$. Suppose that 
$$M_p(Y,T) \le KT^{\frac{\zeta(p-1)}{p}}.$$
Denote $n = 2^{4T}$. Then trere exists an $n$-point metric space $(X,d_X)$ such that 
$$ c_Y(X,d^\alpha_X) \ge C \frac{1}{K}(\log n)^{\alpha - \frac{1+\zeta(p-1)}{p}}\text{, for every }\alpha \in \Big [\frac{1+\zeta(p-1)}{p},1 \Big ],$$
where $C > 0$ is an absolute constant.
\end{lemma}
Lemma \ref{Hamm} as stated does not  claim that $(X,d_X)$ do not depend on $p$, but the proof is given for $4T$-dimensional discrete Hamming cube, i.e., ${(X,d_X) = (\{0,1\}^{4T},||\cdot||_1)}$.
Applying  Lemma \ref{Hamm} to $Y = \PPP_p(\R^d)$, $\zeta = \frac{\frac{p}{2} - 1}{p - 1}$ and $K = 16d^{\frac{1}{2} - \frac{1}{p}}p^{\frac{1}{2}}$ we obtain Corollary \ref{Distortion}. 
\end{proof}

\begin{proof}[Proof of Corollary \ref{WEucl}(\ref{W2Eucl})]
The proof is based on the following proposition.
\begin{proposition}[\cite{NPSS}, Theorem 1.2]\label{Magic}
For $p \ge 2$ we have $M_2(L_p) \le 4\sqrt{p - 1}$.
\end{proposition}
Euclidean space $\R^d$ is $d^{\frac{1}{2} - \frac{1}{p}}$-equivalent to $\R^d_p$. Hence, $\PPP_p(\R^d)$ is  
$d^{\frac{1}{2} - \frac{1}{p}}$-equivalent to $\PPP_p{(\R^d_p)}$.
Thus, $M_2(\PPP_p(\R^d)) \le d^{\frac{1}{2} - \frac{1}{p}}M_2(\PPP_p{(\R^d_p)})$.

The remaining part of the proof is similar to the proof of Proposition \ref{WassMarkov}.
Proposition \ref{embed} provides us isometries ${\Phi_n:{n^{-\frac{1}{p}}}((\R^d_p)^n_p/S_n) \rightarrow \PPP_p(\R^d_p)}$. We denote by $I_k$ the image of $\Phi_{2^k}$.
We have  $I_k \subset I_{k+1}$ for every $k \in \N$. Since the union $\cup_{k=1}^{\infty}I_k$ is dense in $\PPP_p(\R^d_p)$ it follows that 
$$M_2(\PPP_p(\R^d_p)) = \sup_{k \in \N}M_2(I_k) =  \sup_{k \in \N}M_2((\R^d_p)^k_p/S_k).$$

By Proposition \ref{Markovquotient} we have $M_2((\R^d_p)^k_p/S_k) \le M_2((\R^d_p)^k_p) = M_2(\R^{dk}_p)$. Proposition \ref{Magic} implies that $M_2(\R^{dk}_p) \le 4\sqrt{p - 1}$. Hence, $M_2(\PPP_p(\R^d_p)) \le 4\sqrt{p - 1}$.
\end{proof}

The following example shows that Theorem \ref{ThmCases}(\ref{FSC},\ref{FGC}) does not hold in general for infinitely sheeted coverings and infinite groups of isometries. 

\begin{example}\label{FiniteIsNes}
Consider the $d$-dimensional Hamming cube $\Omega^d$, i.e a set $\{0,1\}^d$ with the $L_1$ metric. For the Markov type constants $M_2(\Omega^d)$ we have
 $$M_2(\Omega^d) \xrightarrow[d \to \infty]{} \infty,$$ see \cite[Section 9.4]{IRP}.

The Hamming cube $\Omega^d$ can be converted to a metric graph $G(\Omega^d)$ by adding edges of length $1$ between every two points $x,y \in \Omega^d: d(x,y) = 1$. Consider
 the universal cover $\widetilde{G}(\Omega^d)$ of $G(\Omega^d)$.  The graph $\widetilde{G}(\Omega^d)$ is a metric tree and consequently $M_2(\widetilde{G}(\Omega^d)) \le 30$, see \cite{NPSS}. 

Thus for a large enough $d$ we have $M_2(\widetilde{G}(\Omega^d)) \le 30 < M_2(G(\Omega^d))$. 
\end{example}

\begin{definition}\label{subm}
Let $X$ and $Y$ be metric spaces. A map $\chi:X \rightarrow Y$ is a \textit{submetry}, iff for every $x \in X$ and every $r > 0$ $$\chi(B(x,r))  = B(\chi(x),r),$$
where $B(x,r)$ denotes the closed ball with center $x$ and radius $r$.  
\end{definition}

The following conjecture suggests a uniform approach to Propositions  \ref{MarkovCovering} and \ref{Markovquotient}.

\begin{conjecture}\label{Conj}
Let $X$ and $Y$ be metric spaces such that there exists a submetry $\chi:X \rightarrow Y$, such that for every $y \in Y$ the set $\chi^{-1}(y)$ is finite. Then $ M_2(X) \ge M_2(Y) $. 
\end{conjecture}

We did not found a proof, or a counter example to this conjecture. But we have an example which shows that our method, i.e., lifting of a Markov walks does not work, see Proposition \ref{CantLift} and Example \ref{ExCantLift}.  

\begin{proposition}\label{CantLift}
There exist finite metric spaces $\widetilde{X}$, $X$, a submetry ${\chi:\wt{X} \rightarrow X}$, a stationary reversible Markov chain $\{Z_t\}_{t = 0}^{\infty}$ on a finite set $S$, and an injective map $f:S \rightarrow X$ such that $f(Z_t)$ does not admit a metric lift along $\chi$.
\end{proposition}

The proof of Proposition \ref{CantLift} occupies the rest of this section. The construction is given in the following example.

\begin{example}\label{ExCantLift}
Let $X = \{x_1,x_2,x_3, x_4\}$, $G_{X}$ a graph with vertex set $X$ and $5$ edges connecting all pairs of vertices except $x_2$ and $x_4$. We consider $X$ as a metric space with metric induced from $G_{X}$, i.e distance between every pair of points except $\{x_2,x_4\}$ equals $1$. And distance between $x_2$ and $x_4$ equals $2$.

Let $\{Z_t\}_{t = 0}^{\infty}$ be a markov chain on the set $S = X = \{x_1,\dots,x_4\}$ with stationary distribution $(\frac{3}{10},\frac{2}{10},\frac{3}{10},\frac{2}{10})$ and transition matrix 
$A =
\begin{bmatrix}
    0  &  \frac{1}{3} & \frac{1}{3} & \frac{1}{3}  \\
    \frac{1}{2}  &  0 & \frac{1}{2} & 0  \\
    \frac{1}{3}  &  \frac{1}{3} & 0 & \frac{1}{3}  \\
    \frac{1}{2}  &  0 & \frac{1}{2} & 0 
\end{bmatrix},
$
and let $f = id$.

Let $\wt{X} = \{\wt{x_1},\dots,\wt{x}_{12}\}$, $G_{\wt{X}}$ a graph with vertex set $X$ and $16$ edges. The first group of edges forms the loop $\wt{x_1},\dots,\wt{x}_{12}$. The second group contains remaining $4$ edges connecting following pairs of vertices $\{\wt{x_3},\wt{x}_1\}$, $\{\wt{x_3},\wt{x}_5\}$, $\{\wt{x_9},\wt{x}_7\}$, $\{\wt{x_9},\wt{x}_{11}\}$.
Again we consider $\wt{X}$ as a metric space with metric induced from  $G_{\wt{X}}$.   

Let $r_4:Z_{+} \rightarrow \{1,2,3,4\}$ be the reminder of a number modulo $4$. Let $\chi:\wt{X} \rightarrow X$ be a map defined by 
$$\chi(\wt{x}_i) = x_{r_4(i)}.$$
Note that $\chi$ is a locally surjective graph homomorphism between $G_{\wt{X}}$ and $G_{X}$. Hence, $\chi$ is a submetry.

\end{example}

\begin{lemma}\label{SimpleChain}
Let $\{\widetilde{Z_t}\}_{t = 0}^{\infty}$ and $\{Z_t\}_{t = 0}^{\infty}$ be stationary reversible Markov chains on finite sets $\widetilde{S}$ and $S$. Suppose that $\wt{Z}$ is a lift of $Z$ along a map ${{\ss:\wt{S} \rightarrow S}}$. Then we have 
$ P^{\wt Z}(\wt{s}_1, \ss^{-1}(s_2)) = P^{Z}(\ss(\wt{s}_1),\{s_2\}),\label{ChainLocTriv}$
for every $\wt{s_1} = \wt{S}$ and every $s_2 \in S$.
\end{lemma}
\begin{proof}
The proof is built around the following equality, which follows from Definition \ref{LiftChain}
\begin{equation}A^{\wt Z}(\ss^{-1}(s_2),\ss^{-1}(s_1), \ss^{-1}(s_2)) = A^{Z}(s_2, s_1,s_2),\label{E1}\end{equation}
where $s_1,s_2 \in S$.

Fix $s_2 \in S$, expanding left and right sides of (\ref{E1}) we obtain 
\begin{equation}\label{SCEQ1}
\sum_{\wt s_1 \in \ss^{-1}(s_1)}({A^{\wt Z}(\ss^{-1}(s_2),\{\wt s_1\})}{P^{\wt Z}(\wt s_1, \ss^{-1}(s_2))}) = {A^{Z}(\{s_2\}, \{s_1\})}{P^{Z}(s_1,\{s_2\})}, \end{equation}
Reversibility of Markov chains implies that ${A^Z(\{s_1\}, \{s_2\}) = A^Z(\{s_2\}, \{s_1\})}$ and  
${A^{\wt Z}(\ss^{-1}(s_2),\{\wt s_1\}) = A^{\wt Z}(\{\wt s_1\}, \ss^{-1}(s_2))}$. Thus, we can rewrite (\ref{SCEQ1}) as 
\begin{equation}\sum_{\wt s_1 \in \ss^{-1}(s_1)}\frac{A^{\wt Z}(\ss^{-1}(s_2),\{\wt s_1\})^2}{A^{\wt Z}(\{\wt s_1\})} = \frac{A^{Z}(\{s_2\}, \{s_1\})^2}{A^{Z}(\{s_1\})}.\label{toCBS} \end{equation}
From Definition \ref{LiftChain} we obtain 

 \begin{equation}A^{Z}(\{s_1\}) = \sum_{\wt s_1 \in \ss^{-1}(s_1)}A^{\wt{Z}}(\{\wt s_1\}),\label{TE1}\end{equation}
\begin{equation}A^{Z}(\{s_2\},\{s_1\}) = \sum_{\wt s_1 \in \ss^{-1}(s_1)}A^{\wt{Z}}(\ss^{-1}(s_2), \{\wt s_1\}).\label{TE2}\end{equation}
Substituting the last two equalities into (\ref{toCBS}) and moving the denominator of the right side to the left we obtain 
\begin{equation}\sum_{\wt s_1 \in \ss^{-1}(s_1)}A^{\wt{Z}}(\{\wt s_1\})\sum_{\wt s_1 \in \ss^{-1}(s_1)}\frac{A^{\wt{Z}}(\ss^{-1}(s_2), \{\wt s_1\})^2}{A^{\wt{Z}}(\{\wt s_1\})} = \Big (\sum_{\wt s_1 \in \ss^{-1}(s_1)}A^{\wt{Z}}(\ss^{-1}(s_2), \{\wt s_1\})\Big )^2.\label{TE3}\end{equation}
This is the equality case of the Cauchy-Schwarz inequality.
Hence there exists a constant $\wt C = \wt C(s_1,s_2)$ such that $$\frac{A^{\wt{Z}}(\ss^{-1}(s_2), \{\wt s_1\})}{{A^{\wt{Z}}(\{\wt s_1\})}} \overset{def}{=} P^{\wt{Z}}(\wt s_1, \ss^{-1}(s_2))  = \wt C,$$ for every $\wt s_1 \in \ss^{-1}(s_1)$.
From (\ref{TE1}) and (\ref{TE2}) it follows that $${\frac{A^Z(\{s_2\}, \{s_1\})}{A^{Z}(\{s_1\})} \overset{def}{=} P^{Z}(s_1, \{s_2\}) = \wt C.}$$
\end{proof}

\begin{lemma}\label{MassLemma}
Let $\{\widetilde{Z_t}\}_{t = 0}^{\infty}$ and $\{Z_t\}_{t = 0}^{\infty}$ be stationary reversible Markov chains on finite sets $\widetilde{S}$ and $S$. Suppose that $\wt{Z}$ is a lift of $Z$ (see Definition \ref{DefLiftMarkovWalk})  along a map ${\ss:\wt{S} \rightarrow S}$. Let $s_1, s_2 \in S$ be such that ${A^{Z}(\{s_1\},\{s_2\}) \ne 0}$. Let $\wt{S}_1 \subset \ss^{-1}(s_1)$, $\wt{S}_2 \subset \ss^{-1}(s_2)$ be such that 
\begin{equation}A^{\wt{Z}}(\wt{S}_1, \ss^{-1}(s_2) \setminus \wt{S}_2) = 0,\label{Req1}\end{equation} 
\begin{equation}A^{\wt{Z}}(\wt{S}_2, \ss^{-1}(s_1) \setminus \wt{S}_1) = 0,\label{Req2}\end{equation} 
Then
$$\frac{A^{\wt{Z}}(\wt{S}_1)}{A^{Z}(\{s_1\}) }= \frac{A^{Z}(\wt{S}_2)}{A^{Z}(\{s_2\})}.$$
\end{lemma}
\begin{proof}
Let $\wt{s}_1 \in \wt{S}_1$, Lemma \ref{SimpleChain} implies that 
$$P^{\wt Z}(\wt{s}_1, \ss^{-1}(s_2)) = P^{Z}(s_1, \{s_2\}),$$  
Using the assumption $(\ref{Req1})$ we can rewrite this equality as 
$${A^{\wt{Z}}(\{\wt{s}_1\}, \wt{S}_2)} = {A^{\wt{Z}}(\wt{s}_1)}\frac{A^{Z}(\{s_1\}, \{s_2\})}{A^{Z}(\{s_1\})},$$
Summing the previous equalities for all $\wt{s_1} \in \wt{S}_1$ we have 
$${A^{\wt{Z}}(\wt{S}_1,\wt{S}_2) = {A^{\wt{Z}}(\wt{S}_1)}\frac{A^{Z}(\{s_1\}, \{s_2\})}{A^{Z}(\{s_1\})},}$$
The same argument shows that 
$${A^{\wt{Z}}(\wt{S}_2,\wt{S}_1) = {A^{\wt{Z}}(\wt{S}_2)}\frac{A^{Z}(\{s_2\}, \{s_1\})}{A^{Z}(\{s_2\})},}$$
Since  ${A^{Z}(\{s_1\},\{s_2\}) \ne 0}$ we obtain 
$$\frac{A^{\wt{Z}}(\wt{S}_1)}{A^{Z}(\{s_1\})} = \frac{A^{\wt{Z}}(\wt{S}_2)}{A^{Z}(\{s_2\})}.$$
\end{proof}


\begin{proof}[Proof of Proposition \ref{CantLift}]
Let $X$, $\wt X$ and $\chi$ be as in Example \ref{ExCantLift}.
By contradiction, suppose there exist $\{\widetilde{Z_t}\}_{t = 0}^{\infty}$ a stationary reversible Markov chains on finite set $\widetilde{S}$ and a map $\widetilde{f}:\widetilde{S} \rightarrow \widetilde{X}$, such that Markov walk $\wt{f}(\wt{Z_t})$ is a metric lift of $f(Z_t)$ along $\chi$. Note that since $f$ is injective the Markov chain $\wt Z_t$ is a lift of a Markov chain $Z_t$ along a map $\ss:\wt S \rightarrow S$ defined by $\ss = f^{-1} \circ \chi \circ \wt f$. 
For $1 \le i \le 4$, $1 \le j \le 12$ we denote $A^Z(x_i)$ by $p_i$ and $A^{\wt Z}(\wt f^{-1}(\wt x_j))$ by $q_j$.

Let $i = 1,\dots,11$, consider $\wt{x}_i$ and $\wt{x}_{i+1}$. By Lemma \ref{MassLemma} applied to $s_1 = x_{r_4(i)}, s_2 = x_{r_4(i+1)}, \wt{S}_1 = \wt{f}^{-1}(\wt{x}_i), \wt{S}_2 = \wt{f}^{-1}(\wt{x}_{i+1})$ we have 
\begin{equation}\frac{q_i}{p_{r_4(i)}}= \frac{q_{i+1}}{p_{r_4(i+1)}} .\label{masseq}\end{equation}
These equalities imply that 
\begin{equation}q_3  = q_1  = q_5 \not = 0.\label{AE}\end{equation}
Lemma \ref{MassLemma} applied to $s_1 = x_{3}, s_2 = x_{1}, \wt{S}_1 = \wt{f}^{-1}(\wt{x}_3), \wt{S}_2 = \wt{f}^{-1}(\{\wt{x}_{1},\wt{x}_{5}\})$ implies that  
$$\frac{q_3}{p_3}= \frac{q_1 + q_5}{p_1}.$$
Since $p_3 = p_1 = \frac{3}{10}$ we have 
$$q_3 = q_1 + q_5.$$
This contradicts (\ref{AE}). 
\end{proof}


\bibliography{circle}
\bibliographystyle{plain}

\end{document}